\documentclass{amsart}
\usepackage[utf8]{inputenc}
\usepackage{amsmath}
\usepackage{amssymb}
\usepackage{amsfonts}
\usepackage{comment}
\usepackage{amsthm}
\usepackage{graphicx}

\usepackage{calc}
\usepackage{accents}

\usepackage{esint}

\newtheorem{theorem}{Theorem}
\newtheorem{lemma}{Lemma}

\numberwithin{equation}{section}

\newcommand{\N}{\mathbb{N}}
\newcommand{\R}{\mathbb{R}}
\newcommand{\HH}{\mathcal{H}}
\newcommand{\eps}{\varepsilon}
\newcommand{\BV}{\mathrm{BV}}
\newcommand{\D}{\mathcal{D}}

\newcommand{\dbtilde}[1]{\accentset{\approx}{#1}}

\DeclareMathOperator{\supp}{supp}
\DeclareMathOperator{\sign}{sign}

\title{Schur property for jump parts of gradient measures}
\author{Krystian Kazaniecki}
\address{Krystian Kazaniecki, Institute of Analysis, Johannes Kepler University Linz, Alternberger Strasse 69, A-4040 Linz, Austria, \and Institute of Mathematics, University of Warsaw, Stefana Banacha 2, 02-097 Warszawa, Poland}
\email{krystian.kazaniecki@jku.at}

\author{Anton Tselishchev}
\address{Anton Tselishchev, Institute of Analysis, Johannes Kepler University Linz, Alternberger Strasse 69, A-4040 Linz, Austria}
\email{celis-anton@yandex.ru}

\author{Micha{\l}  Wojciechowski}
\address{Micha{\l}  Wojciechowski, Institute of Mathematics, Polish Academy of Sciences,  Jana i Jedrzeja \'Sniadeckich 8, 00-656 Warszawa, Poland}
\email{miwoj@impan.pl}

\thanks{This research was supported by the National Science Centre, Poland, and Austrian Science Foundation FWF joint CEUS programme. National Science Centre project no. 2020/02/Y/ST1/00072 and FWF project no. I5231}


\begin{document}
	
	\begin{abstract}
		We consider weakly null sequences in the Banach space of functions of bounded variation $\BV(\R^d)$. We prove that for any such sequence $\{f_n\}$ the jump parts of the gradients of functions $f_n$ tend to $0$ strongly as measures. It implies that Dunford--Pettis property for the space $\mathrm{SBV}$ is equivalent to the Dunford--Pettis property for the Sobolev space $W^{1,1}.$
	\end{abstract}
	
	\maketitle
	
	\section{Introduction}
	
	\subsection{Background and motivation}
	
	The goal of this paper is to study the weak compactness in the Banach space of functions of bounded variation. Our motivation was to better understand the properties of weakly null sequences in this space which by the Eberlein--Smulian theorem determined weak compactness. This is an inevitable step towards establishing Dunford--Pettis property (or its absence), which is our ultimate goal. Since this space is a closed subspace of the space of measures with finite total variation, it has well known description of weak convergence but it turns out that the special gradient structure yields a certain new phenomenon.
	
	Let us describe it without going too deep into details. Roughly speaking, every gradient measure (i.e., the gradient of a function of bounded variation) could be uniquely decomposed as a sum of absolutely continuous, jump (absolutely continuous with respect to Hausdorff measure of codimension one) and Cantor parts (note that these three parts are usually not gradient measures themselves). The main result of this paper (Theorem \ref{main}) states that jump parts of weakly null sequences tend to 0 in norm. This resembles the well known Schur property enjoyed e.g. by the space $\ell^1$ (see \cite{Schur} for the original theorem and also \cite[Theorem III.A.9]{Wojt} or \cite[Theorem 2.3.6]{AlKal} for a more modern presentation).
	It explains why we allow ourselves to extend the term "Schur property" to our context.
	
	The structure of gradient measures and their weak* convergence are well understood (see e.g. the remarkable Alberti's paper \cite{Alberti} and also \cite{PelWoj}, \cite[Chapter 3]{AFP2000}). On the other hand, a little is known about the geometry of the dual to the space of functions of bounded variation. For recent results in this direction see e.g. \cite{BV_cont_hyp, Dual_BV_1, Dual_BV_2}.
	
	Our proof has geometric and combinatorial flavour. An elementary argument reduces the problem to studying the gradient measures with jump parts concentrated on fixed Lipschitz graphs.
	Under the assumption that the statement fails, oscillation of the densities of jump parts appears.
	This in turn yields high oscillation of the functions in the neighbourhood of the aforementioned Lipschitz graphs which leads to a contradiction. We present the proof for the two-dimensional case and pass to the case of arbitrary dimension in the final section.
	
	\subsection{Basic definitions and formulation of the main result}
	
	For an arbitrary domain $\Omega\subset \R^d$ the space $\BV(\Omega)$ is the space of functions $u$ in $L^1(\Omega)$ such that their distributional gradient $D u$ is a (vector-valued) measure. The norm on this space is defined in a following way: $\|u\|_{\BV(\Omega)}=\|u\|_{L^1(\Omega))}+\|Du\|$. Here and everywhere below for any measure $\mu$ (real or vector-valued) notation $\|\mu\|$ stands for its total variation. We will mainly work with the space $\BV(\R^d)$ and we use notation $\BV=\BV(\R^d)$.
	
	We present necessary definitions and facts about $\BV$ functions here (which are essentially well known); all of them are taken from \cite[Chapter 3]{AFP2000}.
	
	The gradient $Du$ of any function $u\in\BV$ can be written as a sum of its absolutely continuous and singular parts: $Du = D^a u + D^s u$. The singular part can be further decomposed as the sum of the Cantor and jump parts. We proceed with the description of this decomposition. 
	
	We denote by $B_\rho (x)$ the ball in $\R^d$ of radius $\rho$ with center at the point $x$. For any vector $\nu$ (say, of unit length) we will also use the following convenient notation for two halves of the ball $B_\rho (x)$:
	\begin{align*}
		B_\rho^+ (x,\nu) = \{y\in B_\rho(x):(y-x)\cdot \nu > 0\};\\
		B_\rho^- (x,\nu) = \{y\in B_\rho(x):(y-x)\cdot \nu < 0\}.
	\end{align*}
	The set $J_u$ is now defined as the set of all approximate jump points of $u$, that is, $x\in J_u$ if there exist numbers $a\neq b$ and the unit vector $\nu$ such that
	\begin{align*}
		\lim_{\rho\to +0}\frac{1}{|B_\rho^+ (x,\nu)|}\int_{B_\rho^+ (x,\nu)} |u(y)-a|\, dy=0; \\ \lim_{\rho\to +0}\frac{1}{|B_\rho^- (x,\nu)|}\int_{B_\rho^- (x,\nu)} |u(y)-b|\, dy=0.
	\end{align*}
	Here the triple $(a, b, \nu)$ is uniquely determined by these conditions up to a permutation of $a$ and $b$ and change of sign of the vector $\nu$. We denote $\nu=\nu_u(x)$, $u^+ (x)=a$, $u^-(x)=b$. The ``jump part'' of $Du$ is defined as $D^j u = (D^s u)|_{J_u}$. We will also call the set $J_u$ a ``jump set'' of a $BV$ function $u$. The following identity holds for a jump part of the gradient (see \cite[Theorem 3.77]{AFP2000}):
	$$
	D^j u = (u^+-u^-)\otimes\nu_u\HH^{d-1}|_{J_u}.
	$$
	
	We say that $u$ has an approximate limit in $x\in\R^d$ if there exists $z\in\R$ such that
	$$
	\lim_{\rho\to 0+} \frac{1}{|B_\rho (x)|}\int_{B_\rho (x)}|u(y)-z|\, dy = 0.
	$$
	The set of all points which do not satisfy this property is called an approximate discontinuity set and denoted by $S_u$. The set $S_u$ is countably $\HH^{d-1}$-rectifiable, i.e. it means that up to $\HH^{d-1}$-negligible set it is covered by $\bigcup_{k\in\N} \Phi_k ([0,1]^{d-1})$ where $\Phi_k:[0,1]^{d-1}\to\R^d$ are Lipschitz functions. Besides that, we have $\HH^{d-1}(S_u\setminus J_u)=0$ (see \cite[Theorem 3.78]{AFP2000}). The Cantor part of the gradient is defined as $D^c u = (D^s u)|_{\R^d\setminus S_u}$. For simplicity, we may assume that $u$ is a ``precise representative'' (for the definition see \cite[p.197]{AFP2000}) of an element of $\BV$, i.e. 
		\[
		u(x)=\lim_{\rho\to 0+} \frac{1}{|B_\rho (x)|}\int_{B_\rho (x)}u(y)\, dy
		\]
		for any $x\in\mathbb{R}^d\backslash S_u$.
		We define the functions $u^+$ and $u^-$ $\HH^{d-1}$-a.e. simply by putting $u^+=u^-:=u$ outside the set $S_u$. Then for all $x\in \R^d\backslash S_u$ and any choice of the unit vector $\nu$ we have we have
		\begin{align*}
			\lim_{\rho\to +0}\frac{1}{|B_\rho^+ (x,\nu)|}\int_{B_\rho^+ (x,\nu)} |u(y)-u^+(x)|\, dy=0; \\ \lim_{\rho\to +0}\frac{1}{|B_\rho^- (x,\nu)|}\int_{B_\rho^- (x,\nu)} |u(y)-u^-(x)|\, dy=0.
	\end{align*}
	
	We will use the following fact: for $u\in\BV(\Omega)$ there exists a sequence of smooth functions $u_n$ such that $u_n\to u$ in $L^1(\Omega)$ and $\|Du_n\|_{L^1(\Omega)}\to \|Du\|$. This convergence is called ``strict convergence'' (see \cite[Theorem 3.9]{AFP2000}).
	
	From above definitions (and the fact that $D u$ vanishes on $\HH^{d-1}$-negligible set $S_u\setminus J_u$) it follows that if $u$ is a $BV$ function then its gradient has the following canonical decomposition:
	\[
	Du=D^a u + D^c u + D^j u.
	\]
	Our main theorem states that if a sequence of functions $f_n$ converges \emph{weakly} in the space $\BV$ then the jump parts of the gradients of these functions converge \emph{strongly} (as measures).

	\begin{theorem}\label{main}
		Let $\{f_n\}_{n\in\N}$ be a sequence of functions in $\BV$. If $\{f_n\}$ converges weakly in $\BV(\R^d)$ to a function $f$ then
		\[
		\lim\limits_{n\rightarrow\infty}\|\left(D^j(f-f_n)\right)\| = 0.
		\]
	\end{theorem}
	
	\textbf{Remark 1.} We would like to point out that considering only the jump part of the gradient in Theorem~1 is crucial: it is not true for the whole singular part of the gradient. Let us denote the standart triadic Cantor set by $\mathcal{C}$; we can also view it as the group $\mathbb{Z}_2^\omega$. Let $r_n$ be the Rademacher functions defined on the Cantor set (we view them as the functions supported on  the Cantor set --- that is, $r_n(x) = 1$ for $x\in\mathcal{C}$ which has a digit $2$ on the $n$-th position of its triadic expansion). We also denote by $\mu$ the Cantor measure and put $R_n(x) = \int_0^x r_n(t)\, d\mu(t)$.

	Now we put $f_n(x,y)=R_n(x)\Phi(y)$ where $\Phi$ is a non-negative smooth function with compact support such that $\int\Phi = 1$. Clearly, 
	$$
	D^s f_n=D^c f_n=r_n e_1 \,d\,\mu \otimes \Phi\,d\lambda,
	$$
	where $d\,\lambda$ is a one-dimensional Lebesgue measure and $e_1$ is an element of the standard basis in $\R^2$. Observe that $\|D^s f_n\|=1$. To check that $f_n$ tends weakly to $0$ in $\BV$ it is enough to check that $f_n$ and $D^a f_n$ tend weakly to $0$ in $L^1(\R^2)$ and $D^s f_n$ tends to $0$ weakly in $L^1(d\,\mu \otimes d\,\lambda)$. Obviously, we have $\|f_n\|_{L^1(\R^2)}\to 0$ and $\|D^a f_n\|_{L^1(\R^2)}\to 0$.
	
	Now we take an $\R^2$-valued function $g=(g_1, g_2)\in L^\infty(d\, \mu\otimes d\, \lambda)$ and observe that 
	\begin{equation}
		\int g D^s f_n \, d\,\mu\, d\, \lambda =\int r_n \widetilde{g}\, d\, \mu,\label{remark}
	\end{equation}
	where $\widetilde{g}(x)=\int g_1(x,y) \Phi(y)\, d\,\lambda(y)$. We can equivalently treat the function $\widetilde{g}$ as a function defined in the group $\mathbb{Z}_2^\omega$. Its Fourier transform belongs to $c_0(\bigoplus_{n\in\mathbb{N}}\mathbb{Z}_2)$ by Riemann--Lebesgue lemma and therefore the right-hand side of the equation \eqref{remark} tends to $0$.
	
	\begin{figure}
		\centering
		\includegraphics[width=0.7\textwidth]{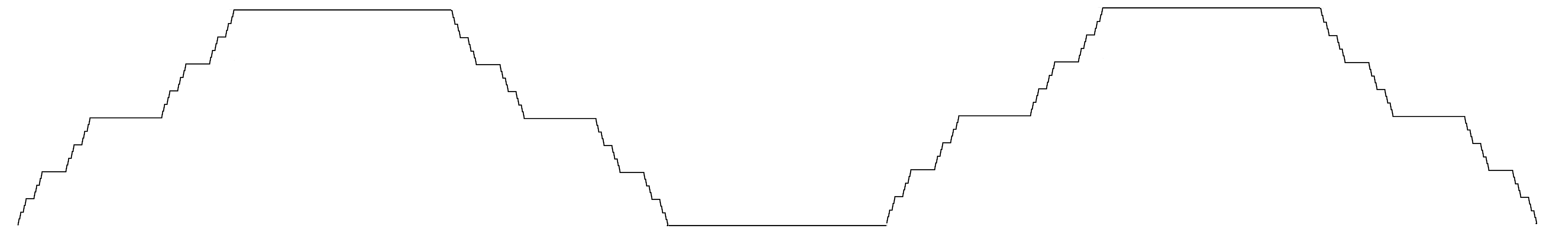}
		\caption{Graph of $R_2(x)$.}
	\end{figure}
	\begin{figure}
		\centering
		\includegraphics[width=0.7\textwidth]{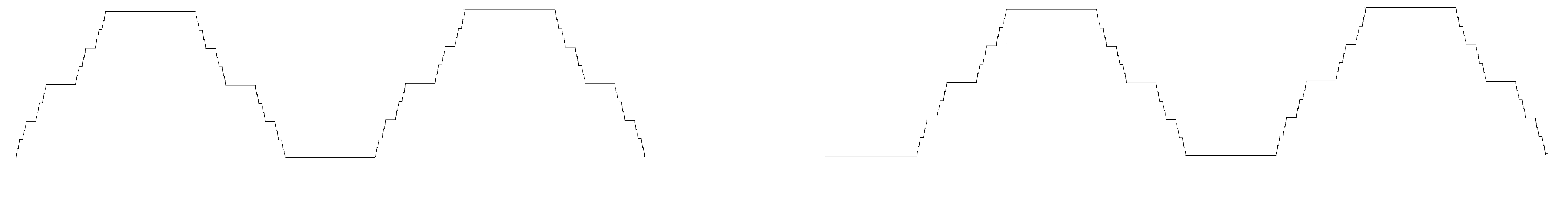}
		\caption{Graph of $R_3(x)$.}
	\end{figure}
	
	\textbf{Remark 2.} Let us now recall that a Banach space $X$ has Dunford--Pettis property (DPP) if for any sequence $x_n$ converging weakly to $0$ in $X$ and any sequence of functionals $x_n^*\in X^*$ converging weakly to $0$ in $X^*$ we have $\lim_{n\to\infty} x_n^*(x_n) = 0$. Besides that, the space $\mathrm{SBV}(\Omega)$ is defined as a subspace of $\BV(\Omega)$ which consists of the functions $f$ such that $D^c f = 0$; for a systematic treatment of this space see \cite[Chapter 4]{AFP2000}.
	
	Assume for simplicity that $\Omega$ is a bounded regular domain. In this case the analogue of Theorem~1 is true for the space $\BV(\Omega)$ (with the same proof). We claim that DPP for the space $\mathrm{SBV}(\Omega)$ is equivalent to DPP for the space $W^{1,1}(\Omega)$. Indeed, since $W^{1,1}(\Omega)$ is a subspace of $\mathrm{SBV}(\Omega)$, DPP for $\mathrm{SBV}(\Omega)$ implies DPP for $W^{1,1}(\Omega)$.
	
	On the other hand, if DPP fails for $\mathrm{SBV}(\Omega)$, there exists a weakly null sequence $f_n\in\mathrm{SBV}(\Omega)$ and a weakly null sequence $\phi_n\in\mathrm{SBV}(\Omega)^*$ such that $\lim_{n\rightarrow \infty}\phi_n(f_n)\neq 0$. By Theorem~1 for the space $\mathrm{SBV}(\Omega)$, $\lim_{n\to\infty}\|D^j f_n\| = 0$. Besides that, by \cite[Theorem 4.6(ii)]{AFP2000} for any $n$ there exists a function $g_n\in \mathrm{SBV}(\Omega)$ such that
		$$D^j g_n= D^j f_n\quad\mbox{and}\quad \|g_n\|_{\mathrm{SBV}(\Omega)}\le C\|D^j f_n\|.$$  
		Since $\lim_{n\to\infty}\|D^j f_n\| = 0$, the sequence 
		\[
		h_n=f_n-g_n
		\]
		is weakly null in $\mathrm{SBV}(\Omega)$. However $h_n\in W^{1,1}(\Omega)$ for every $n=1, 2, \ldots$ and hence this sequence is also weakly null in $W^{1,1}(\Omega)$. The sequence of functionals $\phi_n$ restricted to Sobolev space is weakly null in $W^{1,1}(\Omega)$. Moreover we get
		\[
		\lim_{n\rightarrow \infty}\phi_n(g_n) = \lim_{n\rightarrow \infty}\phi_n(f_n) \neq 0.
		\]
		It yields that $W^{1,1}(\Omega)$ fails DPP.
	
	\textbf{Remark 3.} It is well known that if $\Gamma$ is a Lipschitz surface of codimension one then there exists a trace operator $\mathrm{Tr}:\BV(\R^d)\to L^1(\Gamma)$. It is easy to see that Theorem~1 implies complete continuity of this operator (i.e., it maps weakly convergent sequences to norm-convergent sequences; note that such operator can not be compact --- it follows from the Gagliardo's surjectivity theorem \cite{Gagliardo1957}).
	
	The remaining part of the paper is dedicated to the proof of Theorem~1.
	
	\section{Scheme of proof and certain technical simplifications}
	
	Let us briefly describe the scheme of proof of Theorem~1. It is enough to prove it when the limit function $f$ is equal to $0$. Further, suppose that the statement of the theorem does not hold. Then, extracting a subsequence, we may assume that 
	$\|D^j f_n\|\ge c>0$. After normalization we may also assume that $\|D^j f_n\|=1$. 
	At first, we will work with the jump sets of these functions and do some technical simplifications. In particular, we will show that (up to a small negligible error) we may assume that the sets $J_{f_n}$ have a nice structure: that is, each of them is contained in a finite union of compact Lipschitz graphs. Next, we show that these sets stabilize in a certain sense: that is, there exists such number $N$ that the sets $J_{f_1}, \ldots, J_{f_N}$ cover a ``large part'' of each jump set $J_{f_n}$ with $n>N$. 
	
	In order to simplify our notation, we will concentrate on the case $d=2$. The case of an arbitrary dimension is similar and will be addressed in the final section. After that, since Lipschitz graphs locally behave similar to the intervals we prove our theorem under the assumption that there exists one interval that contains a ``large portion'' of the jump parts of the gradients of functions in our sequence.  Finally, we will use our assumptions in order to find a lot of small sets where the functions highly oscillate. It means that their gradients have large norm which will lead to the desired contradiction.

	Let us start implementing the above strategy of proof. We fix a weakly null sequence $\{f_k\}$ in $\BV$. For simplicity, we introduce notation $J_k:=J_{f_k}$ which will be used from here on. For each set $J_k$ we know that
	$$
	\HH^{d-1} \Big( J_k \setminus \bigcup_{m=1}^\infty \Gamma_{k,m} \Big) = 0
	$$
	where each $\Gamma_{k,m}$ is a compact Lipschitz graph (see \cite[Proposition 2.76]{AFP2000}). Fix any small number $\eps_k>0$. Due to the regularity of the Hausdorff measure, we can find such number $N(k)$ that
	$$
	|D^j f_k| \Big( J_k \setminus \bigcup_{m=1}^{N(k)} \Gamma_{k,m} \Big) < \eps_k.
	$$
	Now we can find the functions $u_k$ in $\BV(\R^d)$ such that $D^j u_k$ is a restriction of the measure $D^j f_k$ to the set $J_k \setminus \bigcup_{m=1}^{N(k)} \Gamma_{k,m}$, i.e. 
		\[
		D^j u_k = D^j f_k|_{J_k \setminus \bigcup_{m=1}^{N(k)} \Gamma_{k,m}}, 
		\] 
	and $\|u_k\|_{\BV}\leq C\eps_k$ (it follows from a consequence of the Gagliardo's surjectivity theorem, see \cite[Theorem 4.6(ii)]{AFP2000}; it is worth noting that, multiplying by appropriate smooth functions with compact supports, we may assume that the supports of $f_n$'s are compact). If we choose the numbers $\eps_k$ tending to $0$ then we see that $f_k-u_k\to 0$ weakly in $\BV$, $\|D^j (f_k-u_k)\|\ge 1/2$ and 
	$$
	\HH^{d-1} \Big(\mathrm{supp}\, D^j (f_k-u_k)\setminus \bigcup_{m=1}^{N(k)} \Gamma_{k,m}\Big) = 0.
	$$
	It means that in the proof of the theorem we may assume that for each function $f_k$ in our sequence the support of the jump part of its gradient is contained (up to $\HH^{d-1}$-negligible set) in a finite union of compact Lipschitz graphs  $\bigcup\limits_{m} \Gamma_{k,m}$. In what follows we will always use this assumption.
	
	\section{Stabilization of jump sets}
	
	Our next goal is to prove the following simple statement about the jump sets of functions in the sequence.
	
	\begin{lemma}
		\label{lem1}
		For any $\eps>0$ there exists $N$ such that for $k>N$ we have
		$$
		|D^j f_k|\Big( J_k\setminus \bigcup_{i=1}^N J_i \Big)<\eps.
		$$ \label{stabilization}
	\end{lemma}
	\begin{proof}
		Assume the opposite. It means that for some $\eps>0$ there exists a subsequence $f_{n_k}$ such that
		$$
		|D^j f_{n_k}|\Big( J_{n_k} \setminus \bigcup_{i=1}^{n_{k-1}} J_i \Big)\ge \eps.
		$$
		
		We would like to construct a bounded linear functional $\phi$ on $\BV$ such that $|\phi(f_{n_k})|\ge \eps$; then we will get a contradiction with the assumption of the weak convergence. 
			Note that for every (scalar-valued) $\HH^{d-1}$-measurable bounded  function $h$ supported on the set $\bigcup\limits_{k,m}\Gamma_{k,m}$  there exists a corresponding functional $\psi_h$ on the space $\BV$ 
			given by the formula 
			$$
			\psi_h (u) := \int_{\bigcup\limits_{k,m}\Gamma_{k,m}} (u^+ - u^-) h \, d\HH^{d-1} = \int_{J_{u}} (u^+ - u^-) h \, d\HH^{d-1}.
			$$
			For each function $u\in\BV$ we have 
			\[
			|\psi_h(u)|\leq \|h\|_{L^{\infty}(\bigcup \Gamma_{k,m})} \|D^j u\|\leq \|h\|_{L^{\infty}(\bigcup \Gamma_{k,m})}\|u\|_{\BV}.
			\]
		
			Now we will construct a particular function $h$. At first we put 
			\[
			h=\mathrm{sign}(f_{n_1}^+ - f_{n_1}^-)\ \text{on the set}\  J_{n_1} 
			\]
		
			We inductively choose the sign $\epsilon_k=\pm 1$ and put
			$$
			h=\epsilon_k\mathrm{sign}(f_{n_k}^+ - f_{n_k}^-)\ \text{on the set}\  J_{n_k} \setminus \bigcup_{j=1}^{k-1} J_{n_j}.
			$$ 
			so that
			\begin{equation}
				\Big|\int_{J_{n_k}} (f_{n_k}^+-f_{n_k}^-) h\, d\HH^{d-1} \Big| \ge \eps.\label{functional}
			\end{equation}
			We indeed can choose such sign $\epsilon_k$:
			since
			\[
			\Big|\int_{J_{n_k} \setminus \bigcup_{j=1}^{k-1} J_{n_j}} (f_{n_k}^+-f_{n_k}^-) h \, d\HH^{d-1} \Big|=|D^j f_{n_k}|\Big( J_{n_k} \setminus \bigcup_{i=1}^{n_{k-1}} J_i \Big)\ge \eps,
			\]
		we only require $\epsilon_k$ to satisfy the condition
		$$
		\sign\Big( \int_{J_{n_k}\cap\bigcup_{i=1}^{k-1} J_{n_i}} (f_{n_k}^+-f_{n_k}^-) h\, d\HH^{d-1} \Big) = \sign\Big( \int_{J_{n_k}\backslash \bigcup_{i=1}^{k-1} J_{n_i}} (f_{n_k}^+-f_{n_k}^-) h\, d\HH^{d-1} \Big).
		$$
		We put $h=0$ outside the set where we have just defined it.
		
		The inequality \eqref{functional} implies that $|\psi_h(f_{n_k})|\ge\eps$ 
		which contradicts the weak convergence.
	\end{proof}
	
	Summarizing the results of the last two subsections, we see that without loss of generality we can assume that $\|D^j f_n\|=1$ and we have a finite family of compact Lipschitz graphs $\{\Gamma_k\}_{k=1}^N$ such that for every $n$
	$$
	|D^j f_n|\Big( \bigcup_{k=1}^N \Gamma_k \Big)\ge 1-\eps.
	$$

	\section{Reasoning for a single interval}
	
	\subsection{Preparation and outline of the proof}
	
	For simplicity, we will now assume that dimension $d$ is equal to $2$; the case of the arbitrary dimension will be discussed in the final section. By application of Lemma~\ref{lem1} for a small number $\eps$ (say, $\eps=\frac{1}{100}$ will suffice) we get that the jump parts of the gradients of functions $f_n$ concentrate on the finite union of Lipschitz graphs. Since locally every Lipschitz graph ``behaves like an interval'', we will further proceed under the assumption that they in fact concentrate just on one interval, i.e., there exists an interval $I$ such that
	\begin{equation}
		|D^j f_n| (I) = \int_I |f_n^+(x,0)-f_n^-(x,0)|\, dx \ge c.\label{grad_jump}
	\end{equation}
	This will allow us to present the main idea of the proof. We will discuss how to pass to a general case in the next section. It is worth noting that the constant $c$ here is actually equal to $\frac{99}{100}$. 
	
	To simplify the notation, assume that $I\subset\R\times\{0\}\subset\R^2$. For $m=1,2,3,\ldots$ consider the partition of $I$ into $2^m$ intervals of equal lengths; we denote the set of intervals in such partition by $\D_m$:
	$$
	\D_m=\{I_{i,m}\}_{i=1}^{2^m}.
	$$

	Consider now the strip $I\times [-\gamma, \gamma]$. We will specify the choice of a small number $\gamma$ in a moment. We have the following inequality for any function $u\in\BV(\R^2)$ and almost every $t\in (0,\gamma)$:
	\begin{align}
		\int_{I} |u^+ (x,0)-u(x,t)|\, dx &\leq \int_{I\times(0,\gamma]} d\, |Du|= |D u|(I\times (0, \gamma]);\label{Fund_Thm_Calc}\\
		\int_{I} |u^- (x,0)-u(x,-t)|\, dx & \;\leq  |D u|(I\times [-\gamma, 0)).
		\label{Fund_Thm_Calc2}
	\end{align}
	Several remarks are in order here. The functions in $\BV(\R^2)$ do not have values at points but they can be defined $\HH^{1}$-a.e. (see \cite[Remark 3.79]{AFP2000}); we will assume that we work with precise representatives of our functions. If the function $u$ is smooth then the inequality \eqref{Fund_Thm_Calc} is obvious: it is just a consequence of fundamental theorem of calculus. For general $\BV$ functions this inequality can be derived from the definition of the values $u^+$ and $u^-$ using the approximation with respect to strict convergence of $\BV$ function $u$ by smooth functions.
	
	For $0<t<\gamma$ put $g_n(x,t)=f_n(x,t)-f_n(x,-t)$; also for $t=0$ put $g_n(x,0)=f_n^+(x,0)-f_n^-(x,0)$.
	By the triangle inequality we have
		\begin{equation*}
			\int_I |g_n(x,0)-g_n(x,t)|\, dx \leq   \int_I |f^+_n(x)-f_n(x,t)|\, dx + \int_I |f^-_n(x)-f_n(x,-t)|\, dx .
		\end{equation*}
		For $u=f_n$ we apply \eqref{Fund_Thm_Calc} and \eqref{Fund_Thm_Calc2} in order to obtain  \begin{equation}
			\int_I |g_n(x,0)-g_n(x,t)|\, dx \leq   |D f_n|(I\times [-\gamma, \gamma]\setminus I \times\{0\}).\label{Calc}
	\end{equation} 
	The following Lemma shows that the right hand side of this inequality can be made arbitrarily small.
	
	\begin{lemma}
		Under our assumptions for every  $\delta>0$ there exists $\gamma > 0$ such that for every $n\in\mathbb{N}$ we have $|D f_n|(I\times [-\gamma, \gamma]\setminus I\times \{0\}) \le \delta$.
	\end{lemma}
	
	We postpone the proof of this Lemma until the end of the present section.
	
	Now we use this Lemma with $\delta=\eps$ and choose $\gamma$ sufficiently small. By \eqref{Calc} we get for almost every $0<t<\gamma$ 
	\begin{equation}
		\int_I |g_n(x,t)-g_n(x,0)|\, dx \le \eps\label{diff_l1}
	\end{equation}
	and hence by \eqref{grad_jump} 
	\begin{equation}
		\label{200420241}
		\int_I |g_n(x,t)|\, dx \ge c-\eps.
	\end{equation}
	
	Let us now briefly describe the main idea of the remaining part of the proof. Since the jump part of the gradient converges weakly to $0$ in $L^1$ on the interval $I\times\{0\}$, it oscillates. Using the inequality \eqref{diff_l1} we will transfer this oscillation to $I\times\{t\}$ for a.e. sufficiently small $t > 0$. This will imply that the total variation of gradients of functions in our sequence is unbounded.
	
	\subsection{Intervals with big averages of absolute values}\label{subs: largeint}
	
	Note that the sequence of functions $f_n^+ - f_n^-$ is weakly null in the space $L^1(I)$. Indeed, for every $h\in L^\infty(I)$ we can define the linear functional 
		$$
		\psi_h(u)=\int_I (u^+-u^-) h\, d\HH^1
		$$
		on the space $\BV$. Since the sequence $\{f_n\}$ is weakly null in $\BV$, we have $\psi_h(f_n) \to 0$. Thus for every $h\in L^\infty(I)$
		\[
		\lim_{n\rightarrow \infty} \int_I (f_n^+-f_n^-) h\, d\HH^1 =0, 
		\]
	and therefore the sequence $\{f_n^+ - f_n^-\}$ is indeed weakly null in $L^1(I)$.
	This in turn implies that it is uniformly integrable with respect to the measure $\HH^1|_I$. Hence we can fix a number $p > 0$ such that for  $A\subset I$ the inequality
	\begin{equation}
		\int_A |g_n(x,0)|\, dx =  \int_A |f_n^+ - f_n^-| \, d\HH^1 \ge\frac{c}{4} \label{int_is_big} 
	\end{equation}
	implies
	\begin{equation}\label{then_meas_is_big}
		\HH^1 (A) \ge 3p.
	\end{equation}


	Let us define the set $L_n^{(m)}(t)\subset \{1,\ldots, 2^m\}$ by the condition
	$$
	i\in L_n^{(m)}(t) \quad \text{iff} \quad \frac{1}{\HH^1(I_{i,m})} \int_{I_{i,m}} |g_n(x,t)|\, dx \ge \frac{c-\eps}{2\HH^1(I)}.
	$$
	Recall that in the beginning of the proof we have chosen the constants $\eps=\frac{1}{100}$ and $c=\frac{99}{100}$. 
		Therefore, it follows that if $i\in L_n^{(m)}(t)$, then
	\begin{equation}
		\label{240420240}
		\quad \frac{1}{\HH^1(I_{i,m})} \int_{I_{i,m}} |g_n(x,t)|\, dx\ge \frac{c}{4\HH^1(I)}.
	\end{equation}
	The following notation is also convenient for us:
	$$
	L_n^{(m)}(t)^*=\bigcup_{i\in L_n^{(m)}(t)} I_{i,m}.
	$$
	Observe that for $i\notin L_n^{(m)}(t)$ we have 
		\[
		\int_{I_{i,m}} |g_n(x,t)|\, dx \le \frac{(c-\eps)\HH^1(I_{i,m})}{2\HH^1(I)}.
		\]
		Since the intervals $\{I_{i,m}\}_{i=1}^{2^m}$ are pairwise disjoint, we get 
		\begin{equation}\label{eq: 200420242}
			\sum_{i\not\in L_n^{(m)}(t)} \int_{I_{i,m}} |g_n(x,t)|\, dx\le \frac{c-\eps}{2}\sum_{i\not\in L_n^{(m)}(t)}\frac{\HH^1(I_{i,m})}{\HH^1(I)}\le\frac{c-\eps}{2}.
	\end{equation}
	Then by \eqref{200420241}, \eqref{eq: 200420242} and the fact that $\{I_{i,m}\}_{i=1}^{2^m}$ is a partition of  $I$, we have:
	\begin{multline*}
		c-\eps\le \int_{I}|g_n(x,t)|\, dx = \sum_{i\in L_n^{(m)}(t)} \int_{I_{i,m}}|g_n(x,t)|\, dx + \sum_{i\not\in L_n^{(m)}(t)} \int_{I_{i,m}}|g_n(x,t)|\, dx\\ \leq \int_{L_n^{(m)}(t)^*} |g_n(x,t)|\, dx + \frac{c-\eps}{2}.
	\end{multline*}
	It yields
	$$
	\int_{L_n^{(m)}(t)^*} |g_n(x,t)|\, dx\ge \frac{c-\eps}{2}.
	$$
	Using the inequality \eqref{diff_l1}, since $L_n^{(m)}(t)^*\subset I$, we deduce that then
	$$
	\int_{L_n^{(m)}(t)^*} |g_n(x,0)|\, dx\ge \frac{c-3\eps}{2}\ge \frac{c}{4}
	$$
	provided that $\eps<\frac{c}{6}$. Therefore, by \eqref{int_is_big} and \eqref{then_meas_is_big} we have $\HH^1(L_n^{(m)}(t)^*)\ge 3p$ and hence
	\begin{equation}
		\# L_n^{(m)}(t) \ge 2^m 3p \HH^1(I)^{-1}.\label{ave_abs}
	\end{equation}
	
	\subsection{Intervals with small averages}
	
	Now we apply Lemma 2 once again, this time with the parameter $\delta =  3\eps_0$ where $\eps_0\le \frac{pc}{300\HH^1(I)}$. By \eqref{Calc} it follows that there exists $\gamma_0$ such that for a.e. $0<t<\gamma_0$ we have
	\begin{equation}
		\int_I |g_n(x,t)-g_n(x,0)|\, dx \le 3\eps_0\label{diff_l1_0}.
	\end{equation}
	
	Now, since the sequence of functions $f_n^+ - f_n^-$ is weakly null in the space $L^1(I)$, for every small $\delta_0>0$ there exists a number $n(\delta_0, m)$ such that for $n>n(\delta_0, m)$ and every $1\le i\le 2^m$ 
	$$
	\frac{1}{\HH^1(I_{i,m})}\Big|\int_{I_{i,m}} g_n(x,0)\, dx\Big|=\frac{1}{\HH^1(I_{i,m})}\Big|\int_{I_{i,m}} (f_n^+(x,0)-f_n^-(x,0))\, dx\Big| < \delta_0.
	$$
	We choose $\delta_0=\eps_0/p$.
	
	Let us denote
	\begin{equation*}
		\alpha_{i,n}^{(m)}(t)=\frac{1}{\HH^1(I_{i,m})} \int_{I_{i,m}} g_n(x,t)\, dx.
	\end{equation*}
	Then for sufficiently big $n$ we have $|\alpha_{i,n}^{(m)}(0)|<\frac{\eps_0}{p}$. 
	We define the set $K_n^{(m)}(t)\subset\{1,\ldots, 2^m\}$ by the following condition:
		\begin{equation*}
			i\in K_n^{(m)}(t) \quad \text{iff}\quad  |\alpha_{i,n}^{(m)}(t)|\leq \frac{\eps_0}{p}+\frac{3\eps_0}{p}=\frac{4\eps_0}{p}
		\end{equation*}
		For $i\in K_n^{(m)}(t)$ it is clear, by the choice of $\eps_0$, that
		\begin{equation}
			\label{240420241}
			|\alpha_{i,n}^{(m)}(t)|\leq \frac{c}{75\HH^1(I)}.
	\end{equation}
	Since for $i\not\in K_n^{(m)}(t) $ we have $|\alpha_{i,n}^{(m)}(t)| \ge \frac{4\eps_0}{p}$, it follows that for such $i$ 
	\[
	|\alpha_{i,n}^{(m)}(t) - \alpha_{i,n}^{(m)}(0)|\ge|\alpha_{i,n}^{(m)}(t)| - |\alpha_{i,n}^{(m)}(0)| \ge \frac{3\eps_0}{p}.
	\]
	Thus, we get
		\begin{equation}\label{eq: 200420243}
			\sum_{i\notin K_n^{(m)}(t)} |\alpha_{i,n}^{(m)}(t)-\alpha_{i,n}^{(m)}(0)|\ge\frac{3\eps_0}{p}(2^m-\# K_{n}^{(m)}(t)).   
	\end{equation}
	Now we will prove an upper estimate for the left-hand side of this inequality. Since $\HH^1(I_{i,m}) = 2^{-m} \HH^1(I)$ for every $1\le i\le 2^m$, we have 
		\[
		\begin{split}
			\int_{I_{i,m}} |g_n(x,t)-g_n(x,0)| \, dx&\ge \Big|\int_{I_{i,m}} g_n(x,t)-g_n(x,0)\, dx \Big|
			\\ &= 2^{-m} \HH^1(I)\;|\alpha_{i,n}^{(m)}(t)-\alpha_{i,n}^{(m)}(0)|.
		\end{split}
		\]
	Therefore
		\[
		\begin{split}
			\int_I |g_n(x,t)-g_n(x,0)|\, dx&\ge \sum_{i\notin K_n^{(m)}(t)} \int_{I_{i,m}} |g_n(x,t)-g_n(x,0)| \, dx
			\\&\ge 2^{-m} \HH^1(I)\sum_{i\notin K_n^{(m)}(t)} |\alpha_{i,n}^{(m)}(t)-\alpha_{i,n}^{(m)}(0)|. 
		\end{split}
		\]
		Combining the above upper estimate with \eqref{eq: 200420243} and \eqref{diff_l1_0}  we get
		$$
		3\eps_0\cdot 2^m \HH^1(I)^{-1} \ge \frac{3\eps_0}{p}(2^m-\# K_{n}^{(m)}(t)).
		$$
	Hence
	\begin{equation}
		\# K_{n}^{(m)}(t) \ge 2^m ( 1-p\HH^1(I)^{-1} ).\label{ave}
	\end{equation}
	
	\subsection{The end of proof}
	
	Denote $S_n^{(m)}(t)=L_n^{(m)}(t)\cap K_n^{(m)}(t)$. Comparing the estimates \eqref{ave_abs} and \eqref{ave} we see that 
	\begin{equation}
		\# S_n^{(m)}(t) \ge \# L_n^{(m)}(t) + \# K_n^{(m)}(t) - 2^m \ge 2^{m+1}p\HH^1(I)^{-1}.\label{good_int}
	\end{equation}
	
	\subsubsection{Proof for smooth functions}
	
	If all our functions $f_n$ were smooth (outside of $I$), then the contradiction would follow almost immediately. Indeed, if $i\in S_n^{(m)}(t)$, then by \eqref{240420240} and \eqref{240420241} we have:
	\begin{align}
		\frac{1}{\HH^1(I_{i,m})}\int_{I_{i,m}} |g_n(x,t)|\, dx &\ge \frac{c}{4\HH^1(I)}\label{one} \\ \Big|\frac{1}{\HH^1(I_{i,m})} \int_{I_{i,m}} g_n(x,t)\, dx\Big| &\le \frac{c}{75\HH^1(I)}.\label{two}
	\end{align}
	It means that the function $x\mapsto g_n(x,t)$ oscillates on the interval $I_{i,m}$. To be more precise, the following inequality follows from here:
	\begin{equation}
		\int_{I_{i.m}} |Dg_n(x,t)|\, dx \ge \frac{71c}{300\HH^1(I)}\ge \frac{c}{5\HH^1(I)} \label{oscill}
	\end{equation}
	The reason for it is elementary for smooth functions: from \eqref{one} we see that there exists a point $x_1\in I_{i,m}$ such that $|g_n(x_1,t)|\ge\frac{c}{4\HH^1(I)}$. Without loss of generality we may assume that $g_n(x_1, t)>0$. Then using \eqref{two} we find a point $x_2\in I_{i,m}$ such that $g_{n}(x_2, t)\le \frac{c}{75\HH^1(I)}$ and then apply the fundamental theorem of calculus to the segment between these two points.
	
	Summing the inequalities \eqref{oscill} over all $i\in S_n^{(m)}(t)$ and applying \eqref{good_int}, we get:
	$$
	\int_I |Dg_n(x,t)|\, dx \ge  \frac{c}{5\HH^1(I)} \# S_n^{(m)}(t)\ge\frac{2^{m+1}pc}{5}\HH^1(I)^{-2}.
	$$
	By Fubini's theorem, we get from here that
	\begin{equation}
		\int_{I\times [\gamma_0/2, \gamma_0]} |Dg_n(x,t)|\, dx\, dt \ge \gamma_0 \frac{pc}{5}2^m\HH^1(I)^{-2}.\label{we_win}
	\end{equation}
	Since $m$ could be arbitrarily large we get that 
	$$
	2\|f_n\|_{\BV} \ge \|g_n\|_{\BV} \to \infty,
	$$
	and it clearly contradicts the assumption of the weak convergence.
	
	In the general case (when functions $f_n$ are non-smooth) we could write a similar proof (using for example \cite[Theorem 5.3.5]{Ziemer} instead of Fubini's theorem); however, it seems to be difficult to generalize such proof to the case of arbitrary dimension. So, we proceed with a different proof which can be easily generalized to any dimension.
	
	\subsubsection{Proof for functions in $\BV$}
	
	In a general case (when functions $f_n$ are not necessarily smooth) we still would like to prove an inequality similar to \eqref{we_win}.
	
	Let us consider the partition of the segment $[0,\gamma_0]$ into intervals 
	$$
	0=t_0<t_1<t_2<\ldots<t_N=\gamma_0,
	$$
	such that 
		$$
		2^{-m-1}\HH^1(I) \le t_{k+1}- t_k \le 2^{-m+1}\HH^1(I), \quad 0\le k\le N-1,
		$$
		and the sets $S_n^{(m)}(t)$ are defined for all $t\in\{t_1,\ldots, t_{N-1}\}$; note that these sets are defined for a.e. $0 < t < \gamma_0$. Observe that 
		\begin{equation}
			N \ge 2^{m-1} \gamma_0 \HH^1(I)^{-1}.\label{To_please_Michal}
	\end{equation}
	
	Now the rectangle $I\times [0,\gamma_0]$ is divided into strips of the form $I\times [t_k, t_{k+1}]$. Consider one such strip which is in turn divided into rectangles (that are actually ``almost squares'' --- the ratio between the lengths of their sides is at most 2) $I_{i,m}\times [t_k, t_{k+1}]$. Let us define yet another set of indices $G_n^{(m)}(t_k)$ as follows:
	\begin{equation}
		G_n^{(m)}(t_k) = \Big\{i\in S_{n}^{(m)}(t_k): \, |Dg_n|(I_{i,m}\times [t_k, t_{k+1}])\leq 2^{-m}\frac{c}{100}\Big\}.\label{good_squares}
	\end{equation}
	Now we estimate the size of the set $S_n^{(m)}(t_k)\setminus G_n^{(m)}(t_k)$. Note that for $1\le k\le N-1$
		\[
		\begin{split}
			|Dg_n|(I\times [t_k, t_{k+1}])&\geq\sum_{i\in S_n^{(m)}(t_k)\setminus G_n^{(m)}(t_k) } |Dg_n|(I_{i,m}\times [t_k, t_{k+1}])
			\\&\geq  \frac{2^{-m}c}{100} (\#S_n^{(m)}(t_k)-\# G_n^{(m)}(t_k)).
		\end{split}
		\]
	Since 
	$$
	|Dg_n|(I\times [t_k, t_{k+1}])  \leq |Dg_n|(I\times (0, \gamma_0])\leq 3\eps_0\leq \frac{pc}{100}\HH^1(I)^{-1},
	$$
	we have
		\begin{equation*}
			\#S_n^{(m)}(t_k)-\# G_n^{(m)}(t_k) \le 2^m p\HH^1(I)^{-1}.
	\end{equation*}
	Hence, by \eqref{good_int} we have the following estimate on the size of $G_N^{(m)}(t_k)$: 
		\[
		\# G_N^{(m)}(t_k) \ge 2^mp\HH^1(I)^{-1}.
		\]
	
	Take any number $i\in G_n^{(m)}(t_k)$. We can rewrite inequalities \eqref{one} and \eqref{two} as follows:
	\begin{align*}
		\int_{I_{i,m}} |g_n(x,t_k)|\, dx &\ge \frac{c}{4}2^{-m}, \\ \Big|\int_{I_{i,m}} g_n(x,t_k)\, dx\Big| &\le \frac{c}{75}2^{-m}.
	\end{align*}
	Recall that by a similar argument as in \eqref{Calc} for a.e. $t\in [t_k, t_{k+1}]$ we get
		\[
		\int_I |g_n(x,t_k)-g_n(x,t)|\, dx \leq   |D g_n|(I\times [t_k, t_{k+1}]).
		\]
	
	Using \eqref{good_squares} for $i \in  G_n^{(m)}(t_k)$ and a.e. $t\in [t_k, t_{k+1}]$ we obtain:
	\begin{align*}
		\int_{I_{i,m}} |g_n(x,t)|\, dx &\ge \frac{c}{4}2^{-m}-\frac{c}{100}2^{-m} = \frac{24c}{100}2^{-m}, \\ \Big|\int_{I_{i,m}} g_n(x,t)\, dx\Big| &\le \frac{c}{75}2^{-m}+\frac{c}{100}2^{-m}\leq \frac{3c}{100}2^{-m}.
	\end{align*}
	
	Denote the rectangle $I_{i,m}\times [t_k, t_{k+1}]$ by $Q_{i,m}^{(k)}$. Recall that $$2^{-m-1}\HH^1(I)\le|t_{k+1}-t_k|\le 2^{-m+1}\HH^1(I).$$ Applying Fubini's theorem we get that
	\begin{align}
		\int_{Q_{i,m}^{(k)}} |g_n(x,t)|\, dx\, dt \geq \frac{24c}{100}2^{-2m-1}\HH^1(I)=\frac{12c}{100}2^{-2m}\HH^1(I),\label{ABS_ONE}\\
		\Big| \int_{Q_{i,m}^{(k)}} g_n(x,t)\, dx\, dt \Big|\leq \frac{3c}{100}2^{-2m+1}\HH^1(I)=\frac{6c}{100}2^{-2m}\HH^1(I).\label{ABS_TWO}
	\end{align}
	
	Now we use the Poincar\'{e} inequality (see \cite[Theorem 3.44 and Remark 3.45]{AFP2000}) for each rectangle $Q_{i,m}^{(k)}$; recall that the lengths of its sides are comparable to $2^{-m}\HH^1(I)$ and therefore they are uniformly bi-Lipshitz images of the ball of radius $2^{-m}\HH^1(I)$:
	\begin{multline}
		|Dg_n|(Q_{i,m}^{(k)})\gtrsim 2^m\HH^1(I)^{-1} \int_{Q_{i,m}^{(k)}} \Big|g_n(x,t)-\fint_{Q_{i,m}^{(k)}}g_n\Big|\,dx\,dt \\ \ge  2^m\HH^1(I)^{-1}\Big( \int_{Q_{i,m}^{(k)}} |g_n(x,t)|\,dx\,dt - \Big| \int_{Q_{i,m}^{(k)}} g_n(x,t)\,dx\,dt \Big| \Big) \geq\frac{6c}{100}2^{-m}.\label{POINCARE}
	\end{multline}
	This inequality holds for at least $2^m p\HH^1(I)^{-1}$ rectangles $Q_{i,m}^{(k)}$ (for $i\in G_n^{(m)}(t_k)$) in a single strip and the number of strips is at least $2^{m-1}\gamma_0\HH^1(I)^{-1}$ (see \eqref{To_please_Michal}). Summing all these inequalities, we get:
	\begin{equation}
		|Dg_n|(I\times [0,\gamma_0]) \gtrsim \gamma_0 \frac{6c}{100} 2^m p \HH^1(I)^{-2}.\label{FINAL}
	\end{equation}
	
	As before, it means that $2\|f_n\|_{\BV}\ge \|g_n\|_{\BV}\to\infty$ and we get a contradiction.
	
	\subsection{Proof of Lemma~2}
	It remains only to prove Lemma~2. Recall that this Lemma states that given any $\delta > 0$ we can find sufficiently small $\gamma > 0$ such that \begin{equation*}
			|D f_n|(I\times [-\gamma, \gamma]\setminus I\times \{0\}) < \delta
		\end{equation*}
		for any $n\in\mathbb{N}$.
	Since
		\[
		\begin{split}
			|D f_n|(I\times [-\gamma, \gamma]\setminus I\times \{0\}) &=|D^a f_n|(I\times [-\gamma, \gamma]\setminus I\times \{0\})
			\\&\;+|D^j f_n|(I\times [-\gamma, \gamma]\setminus I\times \{0\})\\&\;+|D^c f_n|(I\times [-\gamma, \gamma]\setminus I\times \{0\})
		\end{split}
		\]
		we can treat the absolutely continuous, jump and Cantor parts of the gradient consecutively.
	
	\subsubsection{Estimate for absolutely continuous parts}
	
	Note that the functions $D^a f_n$ weakly converge to $0$ in $L^1$: indeed, for any $\R^2$-valued function $g\in L^\infty(\R^2;\,\R^2)$ the functional
	$$
	f\mapsto \int_{\R^2} g\cdot D^a f
	$$
	is a bounded linear functional on $\BV$ (the symbol ``$\cdot$'' here means the scalar product in $\R^2$). Weak convergence to $0$ implies that the functions $D^a f_n$ are uniformly integrable and hence there exists $\gamma^a > 0$ such that the following estimate 
		\begin{equation}
			\int_{I\times [-\gamma,\gamma]} |D^a f_n|\, dx \le \delta/10 \label{est_abs_cont}
		\end{equation}  holds for every $0<\gamma\leq\gamma^a$ and $n\in\mathbb{N}$.
	
	\subsubsection{Estimate for jump parts} 
	
	Let us turn to the jump parts. We apply Lemma~\ref{stabilization} once again, this time with the parameter $\delta/20$. We get that there exists a compact set $K$ (a finite union of Lipschitz graphs) such that $|D^j f_n|(\R^2 \setminus K) \leq \delta/20$.  We have:
	\begin{multline}
		|D^j f_n|(I\times [-\gamma, \gamma]\setminus I \times\{0\})\leq |D^j f_n| (\R^2\setminus K) + |D^j f_n| ((I\times [-\gamma, \gamma]\setminus I\times\{0\})\cap K)\\ \leq \delta/20 + |D^j f_n| ((I\times [-\gamma, \gamma]\setminus I\times\{0\})\cap K).
	\end{multline}
	Similarly as in subsection \ref{subs: largeint} the sequence of functions $f_n^+ - f_n^-$ is weakly null in the space $L^1(\HH^1|_K)$. Indeed, for every $h\in L^\infty(\HH^1|_K)$ we can define the linear functional 
		$$
		\psi_h(u)=\int_K (u^+-u^-) h\, d\HH^1
		$$
		on the space $\BV$. Since the sequence $\{f_n\}$ is weakly null in $\BV$, we have $\psi_h(f_n) \to 0$. Thus for every $h\in L^\infty(\HH^1|_K)$
		\[
		\lim_{n\rightarrow \infty} \int_K (f_n^+-f_n^-) h\, d\HH^1 =0.
		\]
		Therefore, the sequence of functions $\{f_n^+ - f_n^-\}$ is uniformly integrable with respect to the measure $\HH^1|_K$. Since the measure $\HH^1$ on $K$ is finite, by choosing small positive values of $\gamma$ we can make the quantity
		$$
		\HH^1((I\times [-\gamma, \gamma]\setminus I\times\{0\})\cap K)
		$$
		arbitrarily small. Then, applying uniform integrability of the sequence $
		\{f_n^+ - f_n^-\}$, we can assure that $$|D^j f_n| ((I\times [-\gamma, \gamma]\setminus I\times\{0\})\cap K) < \delta/20.$$ 
We see now that there exists $\gamma^j > 0$ such that for $0 < \gamma \le \gamma^j$ and every $n\in\mathbb{N}$ we have
\begin{equation}
	|D^j f_n|(I\times [-\gamma, \gamma]\setminus I \times\{0\})\leq \delta/10.\label{est_jump}
\end{equation}

\subsubsection{Estimate for Cantor parts}

Finally, we show how to treat Cantor parts. It is more difficult because we can not apply uniform integrability directly.

Assume that the statement of Lemma is false. Thus, for every $\gamma>0$ there exists $n\in\mathbb{N}$ such that
	\begin{equation}\label{130420241}
		|D f_n|(I\times [-\gamma, \gamma]\setminus I\times \{0\}) \geq \delta.    
	\end{equation}

We put $\gamma_1=\min\{\gamma^a,\gamma^j\}$. For every positive  $\gamma\leq \gamma_1$ the estimates \eqref{est_abs_cont} and \eqref{est_jump} hold.  
Then by \eqref{130420241}, \eqref{est_abs_cont} and  \eqref{est_jump} there exists a number $n_1$ such that  
$$
|D^c f_{n_1}| (I\times [0,\gamma_1]) \ge 8\delta/10.
$$
Now we find such small number $\gamma_2>0$ that
$$
|D^c f_{n_1}| (I\times [0,\gamma_2]) \le \delta/10.
$$
We continue this process and construct in such a way a subsequence $f_{n_k}$ such that
$$
|D^c f_{n_k}| (I\times [0, \gamma_k]) \ge 8\delta/10
$$
and
\begin{equation}
	|D^c f_{n_k}| (I\times [0,\gamma_{k+1}]) \le \delta/10.\label{cantor_small}
\end{equation}
Note that these two conditions imply the following inequality for every $k$
\begin{equation}
	|D^c f_{n_k}| (I\times [\gamma_{k+1}, \gamma_k]) \ge 7\delta/10.\label{cantor_big}
\end{equation}
Now we define a functional on $\BV$ that will lead us to the contradiction. The main idea is similar to the one we used in the proof of Lemma~1, however here we need to be more careful in order to make a correct definition.

Let us introduce the following measure on $\R^2$:
$$
\mu = \lambda + \sum_{k=1}^\infty \frac{1}{2^k} |D f_{n_k}|,
$$
where $\lambda$ is a Lebesgue measure on $\R^2$. It is a positive finite measure on $\R^2$. Besides that, all gradients of functions $f_n$ are absolutely continuous with respect to this measure which means that there exist such measurable functions $g_{n_k}^a$, $g_{n_k}^c$ and $g_{n_k}^j$ (all of them $\R^2$-valued) that
$$
D^a f_{n_k} = g_{n_k}^a d\mu, \quad D^c f_{n_k} = g_{n_k}^c d\mu,\quad D^j f_{n_k} = g_{n_k}^j d\mu.
$$
Therefore, we can naturally identify each function $f_{n_k}\in \BV$ with an element of the space $L^1(\mu; \R^2)\oplus L^1(\lambda)$; this identification is given by the following map:
$$
f_{n_k}\mapsto (g_{n_k}^a + g_{n_k}^c + g_{n_k}^j, f_{n_k}).
$$
We can extend this map to the whole closed linear span of the sequence $\{f_{n_k}\}$ in $\BV$; besides that, this map is an isometry. In this way we can view $\overline{\mathrm{span}}\{f_{n_k}\}$ as a closed subspace of $L^1(\mu; \R^2)\oplus L^1(\lambda)$ and any pair $(h, \widetilde{h})\in L^\infty(\mu; \R^2)\oplus L^\infty(\lambda)$ defines a functional on $\overline{\mathrm{span}}\{f_{n_k}\}$. By the Hahn--Banach theorem such functional can be extened to the whole space $\BV$.  We will put $\widetilde{h} = 0$ and construct $h$ as follows. 

We will construct a function $h\in L^\infty (\mu; \R^2)$ of norm at most $1$, supported on the set $I\times [0, \gamma_1]$. For $k=1, 2, \ldots$ we put
\[h_k (x) =\left\{
\begin{array}{cc} \frac{g_{n_k}^c(x)}{|g_{n_k}^c(x)|}, & x\in  I\times [\gamma_{k+1}, \gamma_k)\cap\supp(g_{n_k}^c);\\
	0,&\mbox{ otherwise.} \end{array}
\right.
\] 

Clearly, for any $k \ge 1$, since the measures $D^a f_{n_k}$, $D^c f_{n_k}$ and $D^j f_{n_k}$ are mutually singular, the functions $g_{n_k}^a$, $g_{n_k}^c$ and $g_{n_k}^j$ have pairwise disjoint supports. Then by \eqref{cantor_big} we have for $k\ge 1$:
\begin{equation}\label{eq: 070620242}
	\int_{I\times [\gamma_{k+1},\gamma_k)} (g_{n_k}^a + g_{n_k}^c + g_{n_k}^j)\cdot h_k \, d\mu=\int_{I\times [\gamma_{k+1}, \gamma_k)} |g_{n_k}^c|\, d\mu \ge 7\delta/10.
\end{equation}
We put $\eps_1 = 1$ and inductively choose the signs $\eps_k=\pm 1$ for $k\ge 2$ so that
\begin{equation}\label{eq: 070620241}
	\operatorname{sign}\int_{I\times [\gamma_{k+1},\gamma_k)} (g_{n_k}^a + g_{n_k}^c + g_{n_k}^j)\cdot \eps_k h_k \, d\mu =\operatorname{sign}\int_{I\times [\gamma_{k},\gamma_1)} (g_{n_k}^a + g_{n_k}^c + g_{n_k}^j)\cdot \sum_{j=1}^{k-1}\eps_j h_j \, d\mu
\end{equation}
Now we define $
h=\eps_k h_k$ on $I\times [\gamma_{k+1}, \gamma_k)$ for every $k\in\mathbb{N}$.  By \eqref{eq: 070620242} and \eqref{eq: 070620241} we get
$$
\big|\int_{I\times [\gamma_{k+1},\gamma_1)} (g_{n_k}^a + g_{n_k}^c + g_{n_k}^j)\cdot h \, d\mu\big| \ge \int_{I\times [\gamma_{k+1},\gamma_k)} (g_{n_k}^a + g_{n_k}^c + g_{n_k}^j)\cdot h_k \, d\mu \ge 7\delta/10.
$$

By triangle inequality, for any $k\ge 1$ we have
\begin{multline*}
	\Big|\int_{\R^2} (g_{n_k}^a + g_{n_k}^c + g_{n_k}^j) \cdot h\, d\mu\Big| \ge \Big|\int_{I\times [\gamma_{k+1},\gamma_1)} (g_{n_k}^a + g_{n_k}^c + g_{n_k}^j)\cdot h\, d\mu \Big| \\ 
	- \Big( \int_{I\times [0,\gamma_{k+1})} |g_{n_k}^a| |h|\, d\mu + \int_{I\times [0,\gamma_{k+1})} |g_{n_k}^c| |h|\, d\mu + \int_{I\times [0,\gamma_{k+1})} |g_{n_k}^j| |h|\, d\mu \Big).
\end{multline*}
Each of three summands on the second line here do not exceed $\delta/10$: this is guaranteed by inequalities \eqref{est_abs_cont}, \eqref{est_jump}, \eqref{cantor_small} and the fact that $|h|\le 1$ $\mu$-almost everywhere. Summarizing, for $k\ge 1$
$$
\Big|\int_{\R^2} (g_{n_k}^a + g_{n_k}^c + g_{n_k}^j) \cdot h\, d\mu\Big|\ge 4\delta/10.
$$

As we mentioned before, we can extend the functional given by the function $h$ to the whole space $\BV$. Thus, we get a functional $\phi\in\BV^*$ such that $|\langle f_{n_k}, \phi \rangle|\ge 4\delta/10$. This contradicts the weak convergence, and the Lemma is proved.

\section{General case: passing from an interval to Lipschitz graphs and proof in higher dimensions}

\subsection{Lipschitz graphs}

In the previous section we presented the proof in a special case when a Lipschitz graph is replaced by a single interval. Now we address this issue. It is standard: the main idea is simply to ``straighten'' the Lipschitz graph; we will present some details here.

After applying Lemma~1 we obtain a finite number of Lipschitz graphs $\{\Gamma_k\}_{k=1}^N$ such that
$$
|D^j f_n|\Big( \bigcup_{k=1}^N \Gamma_k \Big)\ge 1-\eps.
$$
It means that there exists a number $l\in \{1, 2, \ldots, N\}$ and a subsequence  $f_{n_k}$ such that
$$
|D^j f_{n_k}| (\Gamma_l)\ge \frac{1}{2N}.
$$

Put $W=(0,1)\times(-1,1)$. Since $\Gamma_l$ is a compact Lipschitz graph, there exists a neighborhood $U$ of $\Gamma_l$ and a bijective bi-Lipschitz function $\varphi:\overline{W}\to \overline{U}$ such that $\varphi(I\times\{0\})=\Gamma_l$ where $I\subset [0,1]$ is an interval. We put 
$$
V=\varphi(I\times(-\tau, \tau))
$$ 
for some $\tau<1$.

Let $\Psi$ be a smooth function on $\mathbb{R}^2$ which is equal to $1$ on $V$ and to $0$ on $\R^2\setminus U$. The multiplication operator $M_\Psi:BV(\R^2)\to \BV(U)$ given by the formula $M_\Psi(f)=\Psi f$ is bounded and hence weakly continuous. Therefore, the sequence $\widetilde{f}_{n_k}=M_\Psi f_{n_k}$ converges weakly to $0$ in $\BV(U)$.

Now consider the operator $T_\varphi: \BV(U)\to \BV(W)$ given by the formula $T_{\varphi}(f)=f\circ\varphi$ (strictly speaking, this is a composition with $\varphi|_{W}$). This is indeed a bounded operator (see \cite[Theorem 3.16]{AFP2000}). Therefore, the sequence of functions $\dbtilde{f}_{n_k}=T_\varphi \widetilde{f}_{n_k}$ converges weakly to $0$ in $\BV(W)$. Besides that, by \cite[Theorem 3.16]{AFP2000} (since $\varphi$ is bi-Lipschitz) there exists a constant $c>0$ such that $|D\dbtilde{f}_{n_k}|(I\times\{0\})\ge c$. Obviously, since the set $I\times\{0\}$ has Hausdorff dimension 1, we have $|D\dbtilde{f}_{n_k}|(I\times\{0\})=|D^j\dbtilde{f}_{n_k}|(I\times\{0\})$. 
Now we may repeat the proof from the previous section and get a contradiction.

\subsection{Higher dimensions}

Let us now describe how to modify the proof for the space $\BV(\R^d)$. As we mentioned, it is almost the same: we mainly need to modify the constants in the proof for the 2-dimensional case. In the two dimensional case the proof consists of two steps: the interval case and "passing to the general" case. In higher dimensions the proof also consists of these two steps. The second step is verbatim the same. Now we indicate the modifications in the first step.

First of all, we consider a $(d-1)$-dimensional cube $I$ with side length $\ell(I)$ instead of an interval. For every $m$ we divide the cube $I$ into $2^{m(d-1)}$ equal dyadic subcubes (analogue of the family $\mathcal{D}_{m}$). By the same arguments as above we get that the analogue of the formula \eqref{good_int} will be $\# S_n^{(m)}(t) \ge 2^{m(d-1)+1}p\HH^{d-1}(I)^{-1}.$ Next, the inequalities \eqref{ABS_ONE} and \eqref{ABS_TWO} in $d$ dimensions get the form: 
\begin{align*}
\int_{Q_{i,m}^{(k)}} |g_n(x,t)|\, dx\, dt \geq \frac{12c}{100}2^{-md}\ell(I);\\
\Big| \int_{Q_{i,m}^{(k)}} g_n(x,t)\, dx\, dt \Big|\leq \frac{6c}{100}2^{-md}\ell(I).
\end{align*}

Since the constant in Poincar\`{e} inequality on a $d$-dimensional parallelepiped $Q$ with balanced sides is comparable to $\operatorname{diam}(Q)$ (it follows from a simple scaling argument), the inequality \eqref{POINCARE} takes the following form:
\begin{multline*}
|Dg_n|(Q_{i,m}^{(k)})\gtrsim 2^m\ell(I)^{-1} \int_{Q_{i,m}^{(k)}} \Big|g_n(x,t)-\fint_{Q_{i,m}^{(k)}}g_n\Big|\,dx\,dt \\ \ge  2^m\ell(I)^{-1}\Big( \int_{Q_{i,m}^{(k)}} |g_n(x,t)|\,dx\,dt - \Big| \int_{Q_{i,m}^{(k)}} g_n(x,t)\,dx\,dt \Big| \Big) \geq\frac{6c}{100}2^{-m(d-1)}.
\end{multline*}

Since this inequality holds now for at least $2^{m(d-1)} p\HH^{d-1}(I)^{-1}$ parallelepipeds in one strip and the number of strips is (comparable to) $2^m\gamma_0\ell(I)^{-1}$, we get the following analogue of the final estimate \eqref{FINAL}:
\begin{equation*}
|Dg_n|(I\times [0,\gamma_0]) \gtrsim \gamma_0 \frac{6c}{100} 2^m p \, \ell(I)^{-d}.
\end{equation*}
As before, this yields a contradiction.

Finally, passing to the general case of Lipschitz graphs is the same.

\bibliographystyle{plain}
\bibliography{schur.bib}
\end{document}